\documentclass[12pt]{amsart}

\usepackage{
amsfonts,
latexsym,
amssymb,
}

\newcommand{\labbel}{\label}

\newtheorem{theorem}{Theorem}[section]
\newtheorem{lemma}[theorem]{Lemma}

\newtheorem{proposition}[theorem]{Proposition} 
\newtheorem{cor}[theorem]{Corollary} 
\newtheorem{corollary}[theorem]{Corollary}

\newtheorem*{theorem*}{Theorem}
\newtheorem*{corollary*}{Corollary}

\theoremstyle{definition}
\newtheorem{definition}[theorem]{Definition}

\theoremstyle{remark}

\newcommand{\brfrt}{\hspace{0 pt}}

\DeclareMathOperator{\cf}{cf}

\begin{document}
 
\title[Menger property, ultrafilter convergence]
{A characterization of the Menger property by means of ultrafilter convergence}

\author{Paolo Lipparini} 
\address{Dipartimiento di Matematica\\ Viale della Ricerca Scientifica\\II Universit\`a di Roma (Tor Vergata)\\I-00133 ROME ITALY}
\urladdr{http://www.mat.uniroma2.it/\textasciitilde lipparin}

\thanks{We thank Xavier Caicedo and Boaz Tsaban for very stimulating correspondence.}  



\begin{abstract}
We characterize various Menger/Rothberger-related 
properties, and discuss their behavior
with respect to products.
\end{abstract} 
 
\maketitle

Motivated by classical arguments in the theory of ultrafilter convergence, 
we give a characterization of the Menger property and 
of the Rothberger
property  by means of
ultrafilters and filters, respectively.
In this vein, we discuss the behavior
with respect to products of the above notions, and of some 
weaker variants.

A summary of the paper follows.
In Section \ref{ufconv} we briefly recall 
the notion of filter convergence, together with some classical examples,
which furnish the main motivation for the present paper.

In Section \ref{mengeruf} we show that the Menger property
allows a characterization in terms of ultrafilter convergence.
Actually, our methods work also for the notion in which we 
fix a bound for the cardinality of the covers under consideration.
Even more generally, we can also consider more than countably many families 
of covers, and even allow infinite subsets to be selected.
We are thus led to consider a generalized Rothberger notion 
 $R( \lambda, \mu; \mathord{<}\kappa )$
which depends  on three cardinals, and generalizes simultaneously
the Menger property, the Rothberger property, 
the Menger and the Rothberger properties for countable covers, 
as well as $[ \kappa, \mu]$-compactness
(in particular, countable compactness, initial $\mu$-compactness, Lindel\"ofness and 
final $\kappa$-compactness).  See Definition \ref{threecardmenger}.

In Section \ref{prodthm} 
we study preservation and non preservation 
under products of the 
generalized Rothberger notion 
 $R( \lambda, \mu; \mathord{<}\kappa )$,
establishing a strong connection with 
preservation/non preservation of 
$[ \kappa, \mu]$-\brfrt compactness. 
For example, we get that
 every product of members of some family $\mathcal F$ 
of topological spaces satisfies
the Menger property for countable covers 
if and only if  
 every product of members of $\mathcal F$ is
countably compact.

Finally, in Section \ref{rothbsec} we deal 
with particular cases strongly resembling the
 Rothberger property (in the sense that only one element can be selected
from each cover). In this case, the characterization 
in terms of convergence
involves
filters which are not maximal, and this fact can be seen 
as the one ``responsible'' for non preservation under products.

The sections of the paper are rather independent and,  
apart from some comments and with the exceptions mentioned below, 
 can be read in any order.
All sections use the (classical) definitions recalled in the first two 
paragraphs of Section \ref{ufconv}.
Sections \ref{mengeruf}-\ref{rothbsec}
rely on Definition \ref{threecardmenger}.
Finally, Section \ref{rothbsec} depends  on Lemma \ref{lemr}
and Proposition \ref {mengvslm}.

\section{Some facts about ultrafilter convergence} \labbel{ufconv} 

Many topological properties have been characterized in terms
of ultrafilter convergence. 
Recall that if $I$ is a set, $(x_i) _{i \in I} $ 
is an $I$-indexed sequence  of elements of some topological space  $X$,
and $F$  is a filter over $I$, 
then a point $x \in X$ is said to be an
 \emph{$F$-limit point} 
 of the sequence
$(x_i) _{i \in I} $
 if 
 $\{ i \in I \mid x_i \in U\} \in F$,
for every open neighborhood $U$ of $x$ (Choquet \cite[Section IV]{Ch}, Katetov \cite{Kat}).
Notice that if $X$ is $T_2$, then a sequence has at most one $F$-limit point;
this is the reason for the alternative
expression \emph{$F$-convergent sequence}. 
In this note we shall not actually 
need unicity of $F$-limits, therefore  we shall
assume no separation axiom.

Among the many properties which have been 
considered,   definable in terms
of ultrafilter convergence, a classical one
asks  that \emph{every} sequence has an $F$-limit point, for some given $F$.
In details, a topological space $X$ is \emph{$F$-compact},
for some filter $F$ over $I$,  if every $I$-indexed sequence  of elements of $X$ has some limit point in $X$ (Bernstein \cite{Be} for ultrafilters over $ I= \omega$,
strongly motivated by Robinson non-standard analysis
\cite{Ro}; Saks \cite{Sa} for ultrafilters over cardinals).
$F$-compactness has the pleasant property  of being preserved under taking products.

Interesting topological properties arise from conditions asserting that a topological space is $D$-compact,
for all ultrafilters $D$ in some particular  class of ultrafilters.
For example, a regular topological space $X$ is 
  \emph{$ \lambda $-bounded}
 (i.~e., the closure of every  subset of cardinality $\leq \lambda$ is compact)
if and only if $X$ is $D$-compact, for every ultrafilter $D$ 
over $\lambda$.
In particular, a topological space is compact if and only if 
it is $D$-compact, for every ultrafilter $D$.
See \cite[Theorem 2.9, and Section 5]{Sa}. 

A much broader range of applications of ultrafilter convergence
has been discovered shortly after the appearance of 
Bernstein paper.
For example, it follows from Ginsburg and Saks \cite[p. 404]{GS}  that a topological space is countably compact if and only if,
for every sequence $(x_i) _{i \in \omega } $
of elements of $X$, there exists some ultrafilter
$D$  uniform over 
$ \omega$ such that  
the sequence has some $D$-limit point in $X$.
Here the ultrafilter is not fixed in advance, but, in general,
it depends on the sequence.
The above characterization of countable compactness 
is the key to the result, 
also due to Ginsburg and Saks \cite[Theorem 2.6]{GS},
that all powers of some space $X$ are countably compact
if and only if $X$ is $D$-compact, for some ultrafilter $D$
uniform over $ \omega$.
Similar characterizations, and product theorems as well, 
appear in \cite{Sa} for  general types of accumulation properties,
including, in particular, initial $\mu$-compactness,
and, in equivalent form, $[ \mu, \mu]$-compactness,
for $\mu$ regular.
Saks' results have been subsequently extended to
$[ \kappa, \mu]$-compactness
by Caicedo \cite{Ca}. Caicedo's treatment has also the
advantage of using a single ultrafilter, rather than a family of ultrafilters.

 A comprehensive survey of
earlier results on the subject can be found in Vaughan  \cite{Vahand} 
and Stephenson \cite{Stehand},
together with many related notions and results.
 See Ko{\v{c}}inac and Garc{\'{\i}}a-Ferreira \cite[in particular, Section 3]{KG} for a survey of additional
 results, and Lipparini \cite{tproc2,cmuc,sssr} for even more general treatments
 of the subject.

\section{The Menger property in terms of ultrafilter convergence} \labbel{mengeruf} 

In this section we present a characterization 
of the Menger property
along the lines described in the above section. 
See, e.~g., Scheepers \cite{Sch}, 
Tsaban \cite{Ts}  and Ko{\v{c}}inac \cite{Koc} for information and
references about the Menger and related properties.

Our proofs become  cleaner if we explicitly settle
 the cardinalities of the families of covers under consideration.
Thus we are lead to consider the following general
property which depends  on 3 cardinal parameters (and which might also have independent interest).

\begin{definition} \labbel{threecardmenger}
For $\lambda$, $\mu$, and $\kappa$ nonzero cardinals, let us say that  a topological space $X$
satisfies 
 $R( \lambda, \mu; \mathord{<}\kappa )$,
short for \emph{$X$ satisfies the $\mathord{<}\kappa$-Rothberger property for
sequences of $\lambda$ covers of cardinality $\leq\mu$}, 
if, for every sequence $(\mathcal U_ \alpha ) _{ \alpha \in \lambda } $
of open covers of $X$, such that $| \mathcal U_ \alpha | \leq \mu$,
for every $\alpha \in \lambda $, there are subsets 
$\mathcal V_ \alpha
\subseteq \mathcal U_ \alpha$ ($\alpha \in \lambda $)  
such that 
$\bigcup_{ \alpha \in \lambda } \mathcal V_ \alpha$
is a cover of $X$, and 
$| \mathcal V_ \alpha| <\kappa $, for every $\alpha \in \lambda $.

We shall write 
$R( \lambda, \infty; \mathord{<}\kappa )$
 if we put no restriction on the cardinality of the $\mathcal U_ \alpha $'s, 
and we shall write 
$R( \lambda, \mu; \kappa )$ for
$R( \lambda, \mu; \mathord{<}(\kappa^+) )$.

Thus, in the above notations, the \emph{Menger property}
is  
$R( \omega ,  \infty; \mathord{<} \omega )$, and 
the \emph{Rothberger property}
is  
$R( \omega  , \infty; 1 )$.
Notice that in \cite{cmuc} we used a
nonstandard  terminology,
calling a space ``Menger'' if it satisfies
$R( \omega , \omega , \mathord{<} \omega )$, and 
``Rothberger''  
if it satisfies
$R( \omega , \omega , 1 )$.
Of course, the terminology agrees, say, 
in the class of Lindel\"of   spaces.
Actually, a space is Menger (Rothberger) if and only if 
it is Lindel\"of and satisfies $R( \omega , \omega , \mathord{<} \omega )$
($R( \omega , \omega , 1 )$, respectively).

Here we shall call $R( \omega , \omega , \mathord{<} \omega )$ the 
\emph{Menger property  for countable covers} and 
$R( \omega , \omega , 1)$
the \emph{Rothberger property for countable covers}.
When $ \mu < \infty$, these appear to be the most interesting particular cases of 
  $R( \lambda, \mu; \mathord{<}\kappa )$. The relevance of 
$R( \omega , \omega , \mathord{<} \omega )$ and of 
$R( \omega , \omega , 1)$ has
been pointed out, for example, in \cite{Ts}.

Notice  that, in the above terminology,  the Lindel\"of property 
can be written as  $R( 1 , \infty;  \omega  )$.
There are some trivial relations between $R( \lambda, \mu; \mathord{<}\kappa )$
and $R( \lambda', \mu'; \mathord{<}\kappa' )$, for various cardinals 
$\lambda$, $\lambda'$, $\mu$ \dots, but we shall not need these
here.
\end{definition}

 For  cardinals $\lambda$, $\mu$ and $\kappa$, we let 
$[\mu] ^{\mathord{<} \kappa } $ denote  the set of all  subsets of $\mu$ of cardinality $\mathord{<}\kappa$, and we let
$^ \lambda  ([\mu] ^{\mathord{<} \kappa })  $ denote  the set of all functions from $ \lambda $
to $[\mu] ^{\mathord{<} \kappa } $.
To avoid
complex formulas in subscripts, 
we sometimes shall denote
a sequence 
$(x_i) _{i \in I} $ 
as 
 $ \langle  x_i \mid  i \in I \rangle  $.

\begin{lemma} \labbel{lemr} 
For every topological space $X$, 
and $\lambda$, $\mu$ and  $\kappa$ nonzero cardinals, the following conditions are equivalent. \begin{enumerate}   
\item 
$X$ satisfies  $R( \lambda, \mu; \mathord{<}\kappa )$.
\item
For every sequence $(\mathcal C _ \alpha ) _{ \alpha \in \lambda } $
of families of closed sets of $X$, each family having cardinality $  \leq \mu$,
if, 
for every choice of subfamilies 
$ \mathcal D _ \alpha \subseteq  \mathcal C _ \alpha $
 in such a way  that each 
$ \mathcal D _ \alpha $
has cardinality $ <  \kappa $,
it happens that 
$ \bigcap \{ C \mid C \in \mathcal D _ \alpha , \alpha \in \lambda \} \not= \emptyset  $,
then there is $\bar \alpha \in \lambda $
such that 
$\bigcap \mathcal C _{\bar  \alpha}  \not= \emptyset  $.  
\item
For every sequence $ \langle  x_f \mid  f :\lambda \to  [\mu] ^{\mathord{<} \kappa } \rangle  $ of elements of $X$,
there is $\bar\alpha \in \lambda $
such that 
$\bigcap _{ \beta \in \mu}   E _{ \bar\alpha, \beta   }   \not= \emptyset  $,
where, for  $\alpha \in \lambda $ and
$ \beta  \in \mu$, we put
 $ E _{ \alpha, \beta   }  =
\overline{ \{ x_f \mid f \in I, \beta \in  f( \alpha ) \} } $.
   \end{enumerate} 
\end{lemma} 

 \begin{proof}
The equivalence of (1) and (2) is trivial   by taking complements.

(2) $\Rightarrow $  (3) Let 
 the $ E _{ \alpha, \beta   } $'s be defined as in (3), and let
 $\mathcal E _ \alpha = \{ E _{ \alpha, \beta   } \mid \beta  \in \mu \}  $,
for $\alpha \in \lambda $.
For each $\alpha\in \lambda $, choosing a subfamily $ \mathcal D _ \alpha \subseteq  \mathcal E _ \alpha $
with $ |\mathcal D _ \alpha|  <  \kappa  $ corresponds to choosing 
some $Z_ \alpha  \in [\mu] ^{\mathord{<} \kappa }$ in such a way that
$ \mathcal D _ \alpha =  \{ E _{ \alpha, \beta   } \mid \beta  \in  Z_ \alpha \} $.
If we make such a choice for each $\alpha \in \lambda $,
and we let $f : \lambda \to [\mu] ^{\mathord{<} \kappa }$  
be defined by $f( \alpha ) = Z_ \alpha $, for every $\alpha \in \lambda $, 
then $x_f \in E _{ \alpha, \beta   }$, for 
$\alpha \in \lambda $ and $ \beta \in f( \alpha )$, hence
$ \bigcap \{ C \mid C \in \mathcal D _ \alpha , \alpha \in \lambda \} \not= \emptyset  $.
Applying (2) to the family $(\mathcal E _ \alpha ) _{ \alpha \in \lambda } $,
we get that there is $\bar \alpha \in \lambda $
such that 
$\bigcap \mathcal E _{\bar  \alpha}  \not= \emptyset  $,
that is, (3) holds.  

(3) $\Rightarrow $  (2) Suppose that (3) holds, and that
$(\mathcal C _ \alpha ) _{ \alpha \in \lambda } $ satisfies the premise in (2).
We can write 
$\mathcal C _ \alpha = \{ C _{ \alpha, \beta   } \mid \beta  \in \mu \}  $
(using repetitions of members of $\mathcal C _ \alpha$,
when $|\mathcal C _ \alpha| \mathord{<} \mu$).
For every $ f :\lambda \to  [\mu] ^{\mathord{<} \kappa }$,
we have that 
$ C_f = \bigcap \{ C _{ \alpha, \beta   } \mid  \alpha \in \lambda, \beta \in f( \alpha )\}
\not= \emptyset   $,
by the premise of (2). Pick
some $x_f \in C_f$, for each $f$,
and let $\mathcal E _ \alpha$ and 
 $ E _{ \alpha, \beta   } $ be defined as above.
By (3), there is $\bar\alpha \in \lambda $
such that $\bigcap \mathcal E _{\bar \alpha}  \not= \emptyset  $.
But this implies that $\bigcap \mathcal C _{\bar \alpha}  \not= \emptyset  $,
since, by construction, 
 $ C _{ \alpha, \beta   }  \supseteq E _{ \alpha, \beta   } $,
for every $\alpha \in \lambda $ and $\beta \in \mu$.   
\end{proof}

We say that an ultrafilter $D$ over $^ \lambda  ([\mu] ^{\mathord{<} \kappa })  $  
is \emph{functionally regular} if
there is $\bar\alpha \in \lambda $ 
such that, for every $ \beta \in \mu $,
$A _{\bar \alpha , \beta }= \{ f \in { ^ \lambda ( [\mu] ^{\mathord{<} \kappa })  } \mid \beta  \in f( \bar\alpha )\} \in D$.

\begin{theorem} \labbel{mengthm}
Suppose that  $\mu$ and $\kappa$
are infinite cardinals, $\lambda$
is a nonzero cardinal, and let $I={ ^ \lambda  ([\mu] ^{\mathord{<} \kappa })}$.
Then, for every topological space $X$, the following conditions are equivalent. 
\begin{enumerate}   
\item 
$X$ satisfies $R( \lambda, \mu; \mathord{<}\kappa )$.
\item
For every sequence $\langle  x_f \mid  f \in I  \rangle $ of elements of $X$,
there is $\bar\alpha \in \lambda $
such that 
$\bigcap \mathcal C _{\bar \alpha}  \not= \emptyset  $,
where $\mathcal C _ \alpha = \{ C _{ \alpha, s  } \mid s \in [\mu]^{\mathord{<} \omega } \}  $,
for $\alpha \in \lambda $, and
 $ C _{ \alpha, s  }  =
\overline{ \{ x_f \mid f \in I, s \subseteq f( \alpha ) \} } $,
for $s \in [\mu]^{\mathord{<} \omega }$.
\item
For every sequence $\langle  x_f \mid  f \in I  \rangle $ of elements of $X$, there is a functionally regular ultrafilter $D$ over $I$ such that
$\langle  x_f \mid  f \in I  \rangle $ has some $D$-limit point in $X$.
  \end{enumerate} 
 \end{theorem}

\begin{proof}
The proof combines ideas from \cite[Section 3]{Ca} and \cite[Theorem 5.8]{cmuc}. 

Suppose that (1) holds, and let the sequence  $\langle  x_f \mid  f \in I  \rangle $ be given.
Since $\mu$ is infinite, then $| [\mu]^{\mathord{<} \omega }|= \mu$,
thus $| \mathcal C _ \alpha | \leq \mu$, for every $\alpha \in \lambda $.   
Suppose that, for every $\alpha \in \lambda $, 
$ \mathcal D _ \alpha \subseteq  \mathcal C _ \alpha $,
and  
$ |\mathcal D _ \alpha| <  \kappa $.
This means that,
for every $\alpha \in \lambda $,
there is $Z_ \alpha \subseteq [\mu]^{\mathord{<} \omega }$
such that $| Z _ \alpha |  <  \kappa $,
 and  
$\mathcal D _ \alpha = \{ C _{ \alpha, s  } \mid s \in Z _ \alpha  \}  $.  
Define a function 
$f: \lambda  \to [\mu] ^{\mathord{<} \kappa }$
by 
$f( \alpha )= \bigcup Z_ \alpha $.
Notice that $|\bigcup Z_ \alpha |  <  \kappa $,
since $| Z _ \alpha |  <  \kappa $, $\kappa$ is infinite, and each member of
$| Z _ \alpha |  $ is finite.  
If $\alpha \in \lambda $,
and $s \in Z _ \alpha $,
then $x_f $ belongs to $  C _{ \alpha, s  }$,
thus   
$ x_f \in \bigcap \{ C \mid C \in \mathcal D _ \alpha , \alpha \in \lambda \} $,
hence this last set is nonempty.
Since, by assumption,
$X$ satisfies $R( \lambda, \mu; \mathord{<}\kappa )$,
then, by Lemma \ref{lemr}(2), 
 there is $\bar\alpha \in \lambda $
such that 
$\bigcap \mathcal C _{\bar \alpha}  \not= \emptyset  $.   
Thus (2) is proved.

Now suppose that (2) holds, and that 
$\langle  x_f \mid  f \in I  \rangle $ is a sequence of elements of $X$.
For the $\bar\alpha$ given by (2), we can pick $ x \in \bigcap \mathcal C _{\bar \alpha}   $.   
Thus 
$x \in C _{ \bar\alpha, s  }$,
for every $s \in [\mu]^{\mathord{<} \omega }$.  
For any neighborhood $U$ of $x$ in $X$, let 
$B_U = \{ f \in I \mid x_f \in U\} $, 
and let 
$\mathcal F = \{ B_U \mid U \text{ a neighborhood of } x \} $.
Let 
$\mathcal G = \{ A _{\bar \alpha , \beta } \mid \beta \in \mu \} $, where, 
as in the definition of functional regularity, 
for every $\beta \in \mu$,  we let 
$A _{\bar \alpha , \beta }= \{ f \in I \mid \beta  \in f( \bar\alpha )\}$.
We want to show that $\mathcal F \cup \mathcal G$ has the finite intersection property. 
Notice that  $B_U \cap B_V = B _{U \cap V} $, for every pair 
 $U$, $V$ of neighborhoods of $x$. Moreover, 
$A _{\bar \alpha , \beta_1 } \cap \dots \cap A _{\bar \alpha , \beta_m }
= \{ f \in I \mid s \subseteq  f( \bar\alpha )\}$, for $s= \{ \beta _1, \dots, \beta _m \} $.
 Hence, in order to prove that $\mathcal F \cup \mathcal G$ has the finite intersection property, it is enough to show that 
$B_U \cap A _{\bar \alpha , s}  \not= \emptyset    $, for every
neighborhood $U$ of $x$, and every $s \in [\mu]^{\mathord{<} \omega }$,
where we have put 
$A _{\bar \alpha , s} 
= \{ f \in I \mid s \subseteq  f( \bar\alpha )\}$.
Indeed, 
for every $U$ and $s$ as above, 
since
$x \in C _{ \bar\alpha, s  }$,
there is $f \in I$
such that $s \subseteq  f( \bar\alpha )$,
and $x_f \in U$,
and this means exactly that 
$x_f $ belongs to $ B_U \cap A _{\bar \alpha , s} $,
thus this set is not empty.
We have proved that    
$\mathcal F \cup \mathcal G$ has the finite intersection property.

Thus $\mathcal F \cup \mathcal G$ can be extended to an ultrafilter $D$ over $I$.
Since $D \supseteq \mathcal G$, then $D$ is functionally regular.   
Since $D \supseteq \mathcal F$, then 
$x$ is a $D$-limit point of $\langle  x_f \mid  f \in I  \rangle $.
The implication (2) $\Rightarrow $  (3) is thus proved.

Now assume that (3) holds.
We shall prove that Condition (3) in Lemma \ref{lemr} holds, thus
$X$ satisfies $R( \lambda, \mu; \mathord{<}\kappa )$.
Suppose that 
$\langle  x_f \mid  f \in I  \rangle $ is a sequence of elements of $X$.
Condition (3)  in the present theorem furnishes an
element $x \in X$, an ultrafilter $D$ over $I$,
and an  $\bar\alpha \in \lambda $ 
such that
$x$ is a $D$-limit point of $\langle  x_f \mid  f \in I  \rangle $,
and, for every $ \beta \in \mu $,
$
\{ f \in I \mid \beta  \in f( \bar\alpha )\} \in D$.

We will show that
$x \in E _{ \bar\alpha, \beta   } $,
for every $\beta \in \mu$,
where, as in Lemma \ref{lemr}, 
$ E _{ \bar\alpha, \beta   }  =
\overline{ \{ x_f \mid f \in I, \beta \in  f( \bar\alpha ) \} } $.
If, by contradiction, $x \not\in E _{ \bar\alpha, \beta   } $,
for some $\beta \in \mu$, then there 
is a neighborhood $U$ of $x$ such that 
$U \cap \{ x_f \mid f \in I, \beta \in  f( \bar\alpha ) \} = \emptyset  $,
hence 
$\{ f \in I \mid x_f \in U\}  \cap \{ f \in I \mid \beta  \in f( \bar\alpha )\} = \emptyset$.  
Since $x$ is a $D$-limit point of $\langle  x_f \mid  f \in I  \rangle $,
then $\{ f \in I \mid x_f \in U\} \in D$ and this contradicts 
$\{ f \in I \mid \beta  \in f( \bar\alpha )\} \in D$, since $D$ is a proper filter.
 \end{proof}

\begin{corollary} \labbel{mengcor}
For every topological space $X$, the following conditions are equivalent. \begin{enumerate}  
 \item  
$X$  satisfies the Menger property.  
\item
For every infinite cardinal $\mu$,
and for every sequence $ \langle   x_f \mid  f \in  { ^ \omega  ( [\mu] ^{\mathord{<} \omega  })} \rangle$ of elements of $X$, there is a functionally regular ultrafilter $D$ over ${ ^ \omega   ([\mu] ^{\mathord{<} \omega })}$ such that
the sequence has some $D$-limit point in $X$.
\item
$X$ is Lindel\"of, and, 
for every sequence $ \langle  x_f \mid  f \in  { ^ \omega  ( [ \omega ] ^{\mathord{<} \omega  })}  \rangle$ of elements of $X$, there is a functionally regular ultrafilter $D$ over ${ ^ \omega   ([ \omega ] ^{\mathord{<} \omega })}$ such that
the sequence has some $D$-limit point in $X$.
 \end{enumerate} 
 \end{corollary}

In order to avoid trivial exceptions, when dealing with products, we shall always
assume that all topological spaces under consideration are non empty.

\begin{cor} \labbel{corprodd}
Suppose that  $\mu$ and $\kappa$
are infinite cardinals, $\lambda$
is a nonzero cardinal, and let $I={ ^ \lambda  ([\mu] ^{\mathord{<} \kappa })}$.
Then, for every product $X= \prod _{j \in J} X_j $ of topological spaces, the following conditions are equivalent. 
\begin{enumerate}   
\item 
$X$ satisfies $R( \lambda, \mu; \mathord{<}\kappa )$.
\item
For every choice of sequences $\langle  x _{f,j} \mid  f \in I  \rangle $ in $X_j$
(one sequence for each $j \in J$), 
 there is a functionally regular ultrafilter $D$ over $I$ such that,
for every $j \in J$, the sequence
$\langle  x _{f,j}\mid  f \in I  \rangle $ has some $D$-limit point in $X_j$
(here $D$ is the same for all sequences).
\item 
Every subproduct of $X$ with $\leq 2 ^{2 ^{|I|} } $ factors
(that is, every product $Y= \prod _{j \in J'} X_j $ 
with $J' \subseteq J$ and $|J'| \leq 2 ^{2 ^{|I|} }$)
satisfies $R( \lambda, \mu; \mathord{<}\kappa )$.
 \end{enumerate} 
 \end{cor}

\begin{proof}
The equivalence of (1) and (2) is immediate from Theorem \ref{mengthm} and the fact that a sequence in a product
has a $D$-limit point if and only if each projection onto each factor has a $D$-limit point.

The equivalence of (1) and   (3) is similar to \cite[Theorem 2.3]{Sa}.
Cf. also \cite{Co}. First, notice that
(1) $\Rightarrow $   (3) is trivial. 
We shall complete the proof by showing that if (2) fails, then (3) fails.
Hence suppose that we can choose  sequences $\langle  x _{f,j} \mid  f \in I  \rangle $ in $X_j$, for every $j \in J$, such that, for every 
functionally regular ultrafilter $D$ over $I$, there is
 $j \in J$ such that  the sequence
$\langle  x _{f,j}\mid  f \in I  \rangle $ has no $D$-limit point in $X_j$.
Pick one such $j \in J$ for each functionally regular ultrafilter $D$ over $I$,
and let $J' \subseteq J$ be te set of the  $j$'s chosen in this way.
Since there are $\leq 2 ^{2 ^{|I|} } $ functionally regular ultrafilters over $I$,
we have that $|J'| \leq 2 ^{2 ^{|I|} }$. Now applying the already proved 
equivalence of (1) and (2) to $Y= \prod _{j \in J'} X_j $ we get that
$Y$ does not satisfy $R( \lambda, \mu; \mathord{<}\kappa )$.
\end{proof}

It is not clear whether $ 2 ^{2 ^{|I|} } $
is the best possible bound in Condition (3) in Corollary \ref{corprodd},
in general.
As we shall see in Corollary \ref{corprod} below,
in the parallel situation in which we allow arbitrary repetitions of factors
from a given family, we can indeed get a lower bound.
Of course, in certain particular cases, 
a better bound in Corollary \ref{corprodd} can be obtained.
For example, it is immediate from Proposition \ref{mengvslm} 
below that a product is Menger
if and only if all but finitely many factors are compact, and
the product of the noncompact factors (if any) is Menger.
In particular, a product is Menger if and only if all  countable subproducts are Menger.
As another example, it follows easily from \cite{topappl} 
that a product is Lindel\"of if and only if all subproducts with
$\leq \omega_1$ factors are Lindel\"of.

Some of the  results from \cite{Sa, Ca} 
mentioned in the introduction 
 can be obtained 
 as a particular case of Theorem \ref{mengthm},
by  taking $\lambda=1$.
In fact, condition $R( 1, \mu; \mathord{<} \kappa )$ 
is nothing but a restatement of 
\emph{$[ \kappa, \mu ]$-compactness}.
Moreover, when $\lambda=1$,
an ultrafilter $D$ over  
${ ^ \lambda  ([\mu] ^{\mathord{<} \kappa })} \cong [\mu] ^{\mathord{<} \kappa }$
is functionally regular if and only if 
it \emph{covers} $\mu$, that is,
$\{ z \in [\mu] ^{\mathord{<} \kappa } \mid \beta \in z\} \in D$,
for every $\beta \in \mu$.   
An ultrafilter $D$ over $H$ is \emph{$( \kappa, \mu ) $-regular}
if and only if there is a function $f: H \to [\mu] ^{\mathord{<} \kappa }$ 
such that $f(D)$ covers $\mu$, where 
$f(D)= \{ z \in [\mu] ^{\mathord{<} \kappa } \mid f ^{-1} (z) \in D\} $.
Thus the results in \cite[Section 3]{Ca} are the particular case
of Theorem \ref{mengthm}
when $\lambda=1$.
See \cite{mru} for a survey on regularity of ultrafilters and applications
(notice that $[\mu] ^{\mathord{<} \kappa }$ is denoted by $S _ \kappa ( \mu)$ 
in \cite{mru}).

\section{Product theorems} \labbel{prodthm}

 From Theorem \ref{mengthm}  and the general theory 
developed in \cite{sssr} about 
preservation
of properties under products, we get
that all powers of some space $X$ satisfy
$R( \lambda, \mu; \mathord{<}\kappa )$
if and only if $X$ is $D$-compact, for some 
functionally regular ultrafilter over 
${ ^ \lambda  ([\mu] ^{\mathord{<} \kappa })}$.
However, in this particular case, a stronger result can be obtained
in a direct way.

 We first need an easy proposition. As in Definition \ref{threecardmenger},
it is sometimes convenient to  allow the possibility that
 $\mu= \infty$. Notice that, in this sense, 
$R( 1, \infty; \mathord{<}\kappa )$ is
what is usually called \emph{final $\kappa$-compactness},
which we shall also call  $[ \kappa, \infty ]$-compactness. 
 
If $X=\prod _{ \alpha \in \lambda }  X_ \alpha $, we shall denote by
$\pi _ \alpha $ the natural projection onto the $\alpha ^{\rm th} $ component.  

\begin{proposition} \labbel{mengvslm} 
Let $\lambda$ and $\kappa$ be nonzero cardinals,
and $\mu$ be a nonzero cardinal or $\infty$.
If, for every $\alpha \in \lambda $,
$X_ \alpha $ is a space which is not 
 $[ \kappa, \mu ]$-compact,
then 
$X=\prod _{ \alpha \in \lambda }  X_ \alpha $ does not satisfy
 $R( \lambda, \mu; \mathord{<}\kappa )$.
\end{proposition}

\begin{proof} 
By assumption,
for every $\alpha \in \lambda $,
there is an open  cover $\mathcal W_ \alpha  $
of $X_ \alpha $ such that 
 $|\mathcal W_ \alpha  | \leq \mu$,
and no subset of 
 $\mathcal W_ \alpha  $
of cardinality $<\kappa$ 
 is a cover of  $X_ \alpha  $.
For every $\alpha \in \lambda$,
let $ \mathcal  U_ \alpha $ be the open 
cover of $X$
consisting of all the products of the form
$ \prod _{ \delta  \in \lambda } Y_ \delta  $,
where all $Y_ \delta$'s equal $X_ \delta $,
except for $\delta= \alpha $, in which case  
we require that  $Y _ \delta \in \mathcal W_ \alpha $. 
The family 
 $ (\mathcal  U_ \alpha ) _{ \alpha \in \lambda } $
is clearly a counterexample to $R( \lambda, \mu; \mathord{<}\kappa )$.
We have to show that, for every choice 
of subfamilies 
$\mathcal V_ \alpha
\subseteq \mathcal U_ \alpha$ such that 
$| \mathcal V_ \alpha| <\kappa $ 
for every $\alpha \in \lambda $,
we have that $\bigcup_{ \alpha \in \lambda } \mathcal V_ \alpha$
fails to be a cover of $X$.
Indeed, for every $\alpha \in \lambda $,
since no subset of 
 $\mathcal W_ \alpha  $
of cardinality $<\kappa$ 
 is a cover of  $X_ \alpha  $, 
there is $x_ {\alpha} \in X_ {\alpha}$
such that, whenever 
$x \in X$, and $\pi _ {\alpha} (x) = x_ {\alpha} $,
then $x $ belongs to no member of $  \mathcal V_ {\alpha}$.
By choosing an $x_ \alpha $ as above, for each $\alpha \in \lambda $,
and by taking $x= (x _ \alpha ) _{ \alpha \in \lambda } \in X$,
we get   that $x$ belongs to no member of 
$\bigcup_{ \alpha \in \lambda } \mathcal V_ \alpha$;
thus this last family is not a cover of $X$. 
\end{proof}

In the statement of the next corollary the expression
``product of members of a family $\mathcal F$''  
is always intended in the sense that repetitions are allowed in the 
product, that is, the same space can occur multiple times. 

For $\mu$ a cardinal, we let
$\mu ^{< \kappa } = \sup _{ \kappa ' < \kappa } \mu ^{ \kappa '}  $,
that is, 
$\mu ^{< \kappa } = |[\mu] ^{< \kappa } | $.

\begin{corollary} \labbel{corprod}
Suppose that  $\mu$ and $\kappa$
are infinite cardinals, $\lambda$
is a nonzero cardinal,  and let 
 $\nu= \max \{ \lambda, 2 ^{2 ^{(\mu ^{< \kappa }) } } \} $. 
 For every family $\mathcal F$ of topological spaces, the following conditions are equivalent. 
\begin{enumerate}
   \item 
All products of members of $\mathcal F$  satisfy $R( \lambda, \mu; \mathord{<}\kappa )$.
\item 
All products of $ \leq \nu $ members of $\mathcal F$ satisfy $R( \lambda, \mu; \mathord{<}\kappa )$.
\item
All products of $ \leq \nu $ members of $\mathcal F$ are $[  \kappa, \mu ]$-compact. 
\item
There is a $( \kappa, \mu)$-regular  ultrafilter $D$ 
(which can be chosen over $[\mu] ^{< \kappa } $)
such that every member of
$\mathcal F$ is $D$-compact.
\item
All products of members of $\mathcal F$ are $[ \kappa, \mu  ]$-compact. 
\item
All products of members of $\mathcal F$ satisfy $R( \lambda', \mu; \mathord{<}\kappa )$, for every nonzero cardinal $\lambda'$.
  \end{enumerate}

In particular, every product of members of $\mathcal F$ satisfies
the Menger property for countable covers 
if and only if so does any product of
 $ \leq 2 ^{2 ^{ \omega } } $ factors, if and only if  
 every product of members of $\mathcal F$ is
countably compact.
\end{corollary}

\begin{proof}
The implications 
(1) $\Rightarrow $  (2)
and
(6) $\Rightarrow $  (1) are trivial.

(2) $\Rightarrow $  (3) 
Let $X= \prod _{j \in J} X_j $
be a product of    $ \leq \nu $ members of $\mathcal F$.
Since $\lambda \leq \nu$, and  $\nu$ is infinite, 
then also $X^ \lambda $ is (can be reindexed as)
a product  of 
$ \leq \nu $ members of $\mathcal F$.
By (2), $X^ \lambda $ satisfies $R( \lambda, \mu; \mathord{<}\kappa )$, and,
 by Proposition \ref{mengvslm} with all $X_ \alpha $'s equal to $X$, 
we get that $X$  
is $[\mu, \kappa ]$-compact.
We have proved that (3) holds.

The equivalence of (3), (4) and (5) comes from \cite[Section 3]{Ca}.  
It can be also obtained from \cite{sssr}, and the particular case of Theorem \ref{mengthm} when $\lambda=1$, using the remarks
at the end of Section \ref{mengeruf}.

(5) $\Rightarrow $  (6) follows from the trivial fact that
$[ \kappa , \mu]$-compactness
implies  $R( \lambda', \mu; \mathord{<}\kappa )$,
for every nonzero cardinal $\lambda'$.
 \end{proof}

As in \cite{Ca} or \cite[Remark 2.4]{tapp2},
 there are cases in which the value of $\nu$ in Corollary \ref{corprod}
can be improved. 
Just to state a simple example, we can have  $\nu= \max \{ \lambda, 2 ^{2 ^{ \mu   } } \} $,
when $\mu= \kappa $ is a regular cardinal or, more generally,
when $\cf \mu \geq \kappa $. 

It is trivial to see that if $X$ and $Y$ are topological spaces, $f:X \to Y$ is a continuous and surjective function,
and  $X$
satisfies 
 $R( \lambda, \mu; \mathord{<}\kappa )$,
then $Y$
satisfies 
 $R( \lambda, \mu; \mathord{<}\kappa )$, too.
In particular, if a product 
satisfies  $R( \lambda, \mu; \mathord{<}\kappa )$,
then all subproducts and all factors satisfy it.
Thus from Proposition \ref{mengvslm} 
we get that if some product 
$X=\prod _{ j \in J }  X_ j $ satisfies
$R( \lambda, \mu; \mathord{<}\kappa )$, then all but at most 
$< \lambda $ factors  
are
 $[ \kappa, \mu ]$-compact.
In particular, if a product satisfies the  Menger property
for countable covers,
then all but finitely many factors are countably compact
(this extends \cite[Proposition 1]{Ar}, where a different terminology
is used, and those spaces  we call Menger here are called \emph{Hurewicz} there).

The above remarks
can be improved in combination with results from
  \cite{topappl},
as we shall show in the next proposition.

We let $  \kappa ^{+n}  $ be the $n ^{ \rm th} $ iterated successor
of $\kappa$, that is,  
$  \kappa ^{+0}= \kappa   $, and 
$  \kappa ^{+n+1} =  (\kappa ^{+n})^+  $.

\begin{proposition} \labbel{menglmtapp}
Suppose that 
$\lambda$ and $\kappa$ are infinite cardinals,
$\kappa$ is regular, 
$\mu$  is either an infinite cardinal or $\infty$, 
$ n  \in  \omega $,
and $\mu \geq \kappa ^{+n+1}$.

If some product $X=\prod _{ j \in J }  X_ j$   satisfies
 $R( \lambda, \mu; \kappa ^{+n} )$,
then 
$|\{ j \in J \mid X_j \text{ is not $[ \kappa, \mu]$-compact} \}| 
< \sup \{ \lambda, \kappa ^{+n+1} \}$.
 \end{proposition}

\begin{proof}
Since, for $\lambda' \geq \lambda $,
 $R( \lambda, \mu; \kappa)$ implies 
$R( \lambda', \mu; \kappa)$, it is no loss of generality
to assume that $ \lambda  \geq  \kappa ^{+n+1}$.

By the   remarks before the statement of the proposition, it is enough to prove that
if $Y=\prod _{ \gamma \in \lambda  }  Y_  \gamma $ is
a product of $\lambda$-many spaces, and no $Y _ \gamma $
is  $[ \kappa, \mu]$-compact, then $Y$ does not satisfy 
$R( \lambda, \mu; \kappa ^{+n} )$.
So suppose that $Y$ and the $Y_ \gamma $'s are  as above.
Since a topological space is 
 $[ \kappa, \mu]$-compact if and only if it is both
 $[ \kappa, \kappa ^{+n } ]$-compact and 
 $[ \kappa ^{+n+1} , \mu]$-compact,
and since $\lambda$ is infinite, then
either 
$|\{ \gamma \in \lambda  \mid Y_ \gamma \text{ is not }  
[ \kappa, \kappa ^{+n } ] \text{-\brfrt compact} \}| = \lambda $, or
$|\{ \gamma \in \lambda  \mid Y_ \gamma \text{ is not  }  
[ \kappa ^{+n+1} , \mu]  
\text{-\brfrt compact} \}| = \lambda $. 

In the latter case we have that $Y$ does not satisfy 
$R( \lambda, \mu; \kappa ^{+n} )$, by Proposition \ref{mengvslm}, 
and recalling that  $R( \lambda, \mu; \kappa ^{+n} )$ is the same as 
$R( \lambda, \mu; <\kappa ^{+n+1} )$.

In the former case, for simplicity and
without loss of generality, we can suppose that no 
$Y_ \gamma  $ is  $[ \kappa, \kappa ^{+n } ]$-compact,
hence not $[ \kappa, \kappa ^{+n+1 } ]$-compact.
Partition $\lambda$ into $\lambda$-many classes,
each of cardinality $ \kappa ^{+n+1 } $
(this is possible, since   $ \lambda  \geq  \kappa ^{+n+1}$).
Say, $\lambda= \bigcup _{ \delta \in \lambda } W_ \delta  $.
Then 
$Y=\prod _{ \gamma \in \lambda  }  Y_  \gamma 
\cong
\prod _{ \delta  \in \lambda  } ( \prod _{ \gamma \in W_ \delta }  Y_  \gamma )$.
By \cite[Theorem 23]{topappl}
with $\aleph_ \alpha = \kappa $ 
(here we are using the assumption that $\kappa$ is regular),
we get that, for each $\delta \in \lambda $,
$Z_ \delta =\prod _{ \gamma \in W_ \delta }  Y_  \gamma$
is not   
 $[ \kappa^{+n+1 }, \kappa ^{+n+1 } ]$-compact, 
hence, in particular, not  
$[ \kappa^{+n+1 }, \mu ]$-compact, since $\mu \geq \kappa ^{+n+1}$.
By applying Proposition \ref{mengvslm} to the product
$\prod _{ \delta  \in \lambda  } Z_ \delta $,
we get that $Y \cong \prod _{ \delta  \in \lambda  } Z_ \delta$  does not satisfy 
$R( \lambda, \mu; \kappa ^{+n} )$.
 \end{proof} 

\begin{corollary} \labbel{cpn}
If $\mu \geq \aleph _{n+1}$, and some product satisfies
 $R( \aleph _{n+1}, \mu; \aleph _{n}  )$,
then all but at most $\aleph _{n}$  factors 
are initially $\mu$-compact.
 \end{corollary}

\section{The Rothberger property} \labbel{rothbsec}

We now compare the Menger and the Rothberger properties,
as far as filter convergence is  concerned.

Notice that in the proof of Theorem \ref{mengthm} we made an essential use of the
assumption that $\kappa$ is infinite. Indeed, for $\kappa=2$, say,
for the Rothberger property  $R( \omega  , \infty; 1 )$, a similar characterization
is not possible, at least, not  using ultrafilters.
However, some of the arguments in the proof of Theorem \ref{mengthm}
can be carried over, furnishing a characterization in terms of filters. 
The results from \cite{sssr} do apply also in this 
more general situation, but the fact that the characterization involves only filters
which are not maximal implies that the behavior with respect to products is 
entirely different. See Corollary \ref{corroth} below. 

\begin{proposition} \labbel{rothb}
Suppose that  $\mu$, $\kappa$
and $\lambda$
are nonzero cardinals, and let $I={ ^ \lambda  ([\mu] ^{\mathord{<} \kappa })}$.
Then, for every topological space $X$, the following conditions are equivalent. 
\begin{enumerate}   
\item 
$X$ satisfies $R( \lambda, \mu; \mathord{<}\kappa )$.
\item
For every sequence $\langle  x_f \mid  f \in I  \rangle $ of elements of $X$, there is a filter $F$
 over $I$ such that
  \begin{enumerate}  
  \item  
there is $\bar\alpha \in \lambda $ such that,
 for every $A \in F$ and every $\beta \in \mu$,  
$A \cap A _{\bar \alpha , \beta } \not= \emptyset  $,
where $A _{\bar \alpha , \beta }= \{ f \in I\mid \beta  \in f( \bar\alpha )\}$. 
\item
$\langle  x_f \mid  f \in I  \rangle $ has some $F$-limit point in $X$.
 \end{enumerate} 
\end{enumerate} 
 \end{proposition}

  \begin{proof} 
If $X$ satisfies $R( \lambda, \mu; \mathord{<}\kappa )$,
and $\langle  x_f \mid  f \in I  \rangle $ is a sequence of elements of $X$, 
then, by Lemma \ref{lemr}(3), 
there are $\bar\alpha \in \lambda $, and 
$x \in \bigcap _{ \beta \in \mu}  
\overline{ \{ x_f \mid f \in I, \beta \in  f( \bar\alpha ) \} } $.
For any neighborhood $U$ of $x$ in $X$, let 
$B_U = \{ f \in I \mid x_f \in U\} $; 
then  the filter $F$ generated by
$ \{ B_U \mid U \text{ a neighborhood of } x \} $
witnesses (2). (Notice that
$B_U \cap B_V = B _{U \cap V} $, 
hence any $A \in F$ contains some $B_U$,
hence (2)(a) holds)

Conversely, if (2) holds, then any  $\bar\alpha$ given by (a)
and any point $x$ given by (b) are such that
$x \in \bigcap _{ \beta \in \mu}  
\overline{ \{ x_f \mid f \in I, \beta \in  f(\bar \alpha ) \} } $,
thus Condition (3) in Lemma \ref{lemr} holds.
\end{proof}

Of course, Proposition \ref{rothb} holds also when 
$\mu$ and $\kappa$ are infinite. The main point in Theorem \ref{mengthm} 
is 
the non trivial result
that, in the case when  $\mu$ and $\kappa$ are infinite, we can
 equivalently restrict ourselves to \emph{ultrafilters} satisfying Condition 
(2)(a) in Proposition \ref{rothb}. Indeed, if some ultrafilter $D$ 
over ${ ^ \lambda  ([\mu] ^{\mathord{<} \kappa })}$
satisfies (2)(a),
for $\bar\alpha \in \lambda $, 
then each  $A _{\bar \alpha , \beta }$ belongs to $D$,
since otherwise $A = I \setminus A _{\bar \alpha , \beta } \in D$,
contradicting (2)(a).
Thus an ultrafilter  $D$ satisfies (2)(a) if and only if it is functionally regular. That is, 
Condition (2) in Proposition \ref{rothb}, when restricted to ultrafilters, 
becomes Condition (3)
in Theorem \ref{mengthm}.
 
On the other hand, if  $\kappa = 2$  and $\mu \geq 2$,
in particular, when dealing with approximations $R( \omega , \mu; 1 )$ 
 to the Rothberger property, then there exists no ultrafilter which satisfies 
Condition (2)(a) in Proposition \ref{rothb}. Indeed, arguing as above, such an ultrafilter should
contain $A _{\bar \alpha , \beta }$, for every $\beta \in \mu$. 
However, when $\kappa=2$, we have that 
$A _{\bar \alpha , \beta_1 } \cap A _{\bar \alpha , \beta_2 } = \emptyset $,
for $ \beta _1 \not= \beta _2$, contradicting the property
of being a proper filter. A similar argument applies when $\kappa$ is finite and
  $\mu \geq \kappa $.
In other words, for such $\mu$ and $\kappa$, 
and arbitrary $\lambda$, there is no functionally regular
ultrafilter over $ { ^ \lambda  ([\mu] ^{\mathord{<} \kappa })}$.

The notion of $F$-compactness is usually given only for
 ultrafilters, since no $T_1$ space with more than one point is
$F$-compact, when $F$ is a filter which is not maximal \cite{sssr}.   
However, as Proposition \ref{rothb}  shows,  the notion of an $F$-limit point 
has some interest even when $F$ is a
filter not maximal.
Another example in which general filters are necessary
is sequential compactness, dealt with in \cite{sssr}.

 In  \cite{sssr} we have studied those 
 classes $\mathcal K$ of topological spaces 
which can be characterized by means of filter convergence, 
in the sense that there is a set $I$ and a family
$\mathcal P$ of filters over $I$ such that, 
for every $I$-indexed sequences of elements 
from any space $X \in \mathcal K$,
there is some $F \in \mathcal P$  such that the sequence
$F$-converges to some element of $X$
(here examples of such classes $\mathcal K$ are furnished by
Theorem \ref{mengthm} and Proposition \ref{rothb}).
 It has been proved in  \cite{sssr} that if $\mathcal K$ 
allows a characterization
as above, and the corresponding family $\mathcal P$ 
contains exclusively
filters which are not maximal
(e.~g., as in Proposition \ref{rothb}, but not as in Theorem \ref{mengthm}),
then any space $X$ with the property
that all powers of $X$ belong to $\mathcal K$ 
must be ultraconnected 
(a  topological space is 
 \emph{ultraconnected}  if it 
 does not contain a pair of disjoint nonempty
closed sets). 

In particular, if all powers of some
space $X$ are Rothberger, then
$X$ is ultraconnected, and thus satisfies very few separation properties
(unless it is a one-element space).
In fact, a stronger result follows directly from 
 Proposition \ref {mengvslm}.

\begin{corollary} \labbel{corroth}
If $\lambda$ is a nonzero cardinal and, for every $\alpha \in \lambda $,
$X_ \alpha $ is a space which is not 
ultraconnected,
then 
$X=\prod _{ \alpha \in \lambda }  X_ \alpha $ does not satisfy
 $R( \lambda, 2; 1)$.

In particular, if a product  satisfies the Rothberger property
(even just for countable covers),
then all but a finite number of factors are ultraconnected.
\end{corollary}

\begin{proof}
Immediate from Proposition \ref {mengvslm}, noticing that 
ultraconnectedness is the same as $[2,2]$-compactness
(every two-elements open cover has a one-element subcover).

The last statement follows from the remarks before Proposition \ref{menglmtapp}.
\end{proof}

\end{document}